\documentclass[11pt]{amsart} 
\usepackage{graphicx}
\usepackage{amsmath,amssymb,a4wide} 
\usepackage{mathabx}
\usepackage{color}
\usepackage{subfigure}
\usepackage{mathrsfs}
\usepackage[final]{showkeys}
\usepackage{hyperref}

\usepackage[T1]{fontenc} 

\newtheorem{theorem}{Theorem} 
 
\newtheorem{proposition}[theorem]{Proposition}

\newcommand{\R}{{\mathbb R}} 
\renewcommand{\L}{{\mathbb L}} 
\renewcommand{\H}{{\mathbb H}} 
\newcommand{\C}{{\mathbb C}} 
\newcommand{\N}{{\mathbb N}} 
\newcommand{\U}{{\mathcal U}}
\newcommand{\E}{{\mathbb E}}
\newcommand{\dd}{{\rm d}}
\newcommand{\PP}{{\mathbb P}}
\DeclareMathOperator{\sech}{sech}
\providecommand{\norm}[1]{\left\lVert#1\right\rVert}

\begin{document}

\title{Exponential integrators for the stochastic Manakov equation}

\author{%
{\sc Andr\'e Berg}}
\address{Department of Mathematics and Mathematical Statistics, 
Ume{\aa} University, 90187 Ume{\aa}, Sweden}
\email{andre.berglund@umu.se}
\author{%
{\sc David Cohen}}
\address{Department of Mathematics and Mathematical Statistics, 
Ume{\aa} University, 90187 Ume{\aa}, Sweden and\\
Department of Mathematical Sciences, 
Chalmers University of Technology and University of Gothenburg, 41296 Gothenburg, Sweden}
\email{david.cohen@chalmers.se}
\author{%
{\sc Guillaume Dujardin}}
\address{Inria Lille Nord-Europe and Laboratoire Paul Painlev\'e UMR CNRS 8524,\\
59650 Villeneuve d'Asq Cedex, France}
\email{guillaume.dujardin@inria.fr}

\maketitle

\begin{abstract}
{This article presents and analyses an exponential integrator for the stochastic Manakov equation, a
system arising in the study of pulse propagation in randomly birefringent optical fibers. We first prove
that the strong order of the numerical approximation is $1/2$ if the nonlinear term in the system is globally
Lipschitz-continuous. Then, we use this fact to prove that the exponential integrator has convergence
order $1/2$ in probability and almost sure order $1/2$, in the case of the cubic nonlinear coupling which
is relevant in optical fibers. Finally, we present several numerical experiments in order to support our
theoretical findings and to illustrate the efficiency of 
the exponential integrator as well as a modified version of it.
}
{Stochastic partial differential equations. 
Stochastic Manakov equation. 
Coupled system of nonlinear Schr\"odinger equations. 
Numerical schemes. Exponential integrators. Strong convergence. 
Convergence in probability. Almost sure convergence. 
Convergence rates. }
\end{abstract}

{\small\noindent 

\section{Introduction}

Optical fibers play an important role in our modern communication society 
\cite{agrawal2007nonlinear}.
In order to model the light propagation over long distance in randomly varying birefringent optical fibers, 
the Manakov PMD equation was derived from Maxwell's equations in \cite{Wai-1996}. 
As noted in \cite{garnierm06}, polarization mode dispersion (PMD) is one of the main limiting effects  
of high bit rate transmission in optical fiber links. In addition, the work \cite{garnierm06} proves that 
the asymptotic dynamics of the Manakov PMD equation is given by a stochastic nonlinear evolution equation 
in the Stratonovich sense: the stochastic Manakov equation, see below. 
In the present article, we perform a numerical analysis of this stochastic partial differential equation (SPDE).

We now review the literature on the numerical analysis of the stochastic Manakov equation. 
The work \cite{Gazeau:13}, see also \cite{GazeauPhd}, 
numerically studies the impact of noise on Manakov solitons and soliton 
wave-train propagation by the following time integrators: the nonlinearly implicit Crank--Nicolson 
scheme, the linearly implicit relaxation scheme, and a Fourier split-step scheme. For instance, it is conjectured that, 
in the small-noise regime and over short distances, solitons are not strongly destroyed and are stable. 
Reference \cite{MR3166967}, see also \cite{GazeauPhd}, proves that the order of convergence in probability of 
the Crank--Nicolson scheme is $1/2$. On top of that, it is shown that this numerical integrator preserves the $\L^2$-norm 
as does the exact solution to the stochastic Manakov equation. Furthermore, it is numerically 
observed that the almost-sure order of convergence of the relaxation scheme and the split-step 
scheme is $1/2$. 

The main goal of this article is to present and analyse a linearly implicit exponential integrator for the 
time discretisation of the stochastic Manakov equation. Exponential integrators for the time integration 
of deterministic or stochastic (partial) differential equations are nowadays widely used and studied as witnessed 
by the recent works \cite{MR2413146,MR2536086,MR2578878,MR2652783,MR2995211,MR3047942,MR3033008,MR3353942,
  MR3463447,MR3484400,MR3736655,MR3771721,acqs18,MR3663004} and references therein.
Beside having the same orders of convergence as the nonlinearly implicit Crank--Nicolson
scheme from \cite{MR3166967}, 
the proposed exponential integrators offer additional computational advantages as illustrated below.

\section{Setting and notation}

Let $(\Omega, \mathcal{F},\mathbb{P})$ be a probability space on which a three-dimensional 
standard Brownian motion $W(t):=(W_1(t), W_2(t), W_3(t))$ is defined. We endow the probability space with the 
complete filtration $\mathcal{F}_t$ generated by $W(t)$. 
In the present paper, we consider the nonlinear stochastic Manakov system \cite{MR3166967}
\begin{equation}
  \label{eq:Manakov}
  i{\rm d}X
+\partial^2_x X\, {\rm d}t
+ |X|^2 X\, {\rm d}t
+ i \sqrt{\gamma} \sum_{k=1}^{3} \sigma_k \partial_x X \circ {\rm d} W_k
 = 0,
\end{equation}
where $X=X(t,x)=(X^1,X^2)$ is the unknown function with values in $\C^2$ with $t\geq0$ and $x\in\R$, 
the symbol $\circ$ denotes the Stratonovich product,
$\gamma\geq 0$ measures the intensity of the noise,
$|X|^2=|X^1|^2+|X^2|^2$ is the nonlinear coupling,
and $\sigma_1$, $\sigma_2$ and $\sigma_3$ are the
classical Pauli matrices defined by
\begin{equation*}
  \sigma_1=
  \begin{pmatrix}
    0 & 1\\
    1 & 0
  \end{pmatrix},
  \qquad
  \sigma_2=
  \begin{pmatrix}
    0 & -i\\
    i & 0
  \end{pmatrix},
  \quad
  {\rm and}
  \quad
  \sigma_3=
  \begin{pmatrix}
    1 & 0\\
    0 & -1
  \end{pmatrix}.
\end{equation*}
The mild form of the stochastic Manakov equation reads 
\begin{equation}
\label{mildManakov}
X(t)=U(t,0)X^0+i\int_0^tU(t,s)F(X(s))\,\text ds, 
\end{equation}
where $X^0$ denotes the initial value of the problem, $U(t,s)$ for $t\geq s$ with $s,t\in\R_+$ is the random unitary propagator defined as the unique solution 
to the linear part of \eqref{eq:Manakov}, and $F(X)=|X|^2X$.

Let $p\geq1$. We define $\L^p:=\L^p(\R):=(L^p(\R;\C))^2$ the Lebesgue spaces of functions with values in $\C^2$. 
We equip $\L^2$ with the real scalar product $\displaystyle (u,v)_2=\sum_{j=1}^2\mathrm{Re}\left(\int_\R u_j\overline{v}_j\,\text dx\right)$. 
Further, for $m\in\N$, we denote $\H^m:=\H^m(\R)$ the space of functions in $\L^2$ with their $m$ first derivatives in $\L^2$. 
The norm in $\H^m$ is denoted by $\norm{\cdot}_{m}=\norm{\cdot}_{\H^m}=\norm{\cdot}_{\H^m(\R)}$.

We now recall the local existence and uniqueness result for solutions to \eqref{eq:Manakov} obtained in \cite{MR3024974} (see also \cite{GazeauPhd}).
\begin{theorem}[Theorem~1.2 in \cite{MR3024974}]\label{thmExact}
Consider the initial value $X^0\in\H^1$, then there exists a maximal stopping time $\tau^*(X^0,\omega)$ and a unique solution $X$ (in the probabilistic sense) to 
\eqref{eq:Manakov} such that $X\in C([0,\tau^*[,\H^1)$ $\mathbb{P}$-a.s. Furthermore, the $\L^2$ norm is almost surely preserved: 
$\norm{X(t)}_{\L^2}=\norm{X^0}_{\L^2}$ 
for $t\in[0,\tau^*[$. Moreover, the following alternative holds for the maximal existence time of solutions to \eqref{eq:Manakov}:
$$
\tau^*(X^0,\omega)=+\infty\quad\text{or}\quad \limsup_{t\nearrow\tau^*(X^0,\omega)}\norm{X(t)}_{\H^1}=+\infty.
$$
Finally, if the initial value $X^0$ belongs to $\H^m$ for some $m\geq1$, then the corresponding solution also belongs to $\H^m$ almost surely.
\end{theorem}
As seen above, the $\L^2$ norm of the solution is preserved just as for the deterministic Manakov
equation (i.\,e. \eqref{eq:Manakov} with $\gamma=0$).
Furthermore, as noted by \cite{MR3166967}, the occurrence of blow-up in the stochastic Manakov equation \eqref{eq:Manakov} remains an open question.

For the time discretisation of the stochastic Manakov system \eqref{eq:Manakov}, 
one has to face two issues. First, the linear part of the equation generates a random unitary propagator which 
is not easy to compute exactly. In particular, since the Pauli matrices do not commute, 
it is not the product of the stochastic semi-groups associated to each
Brownian motion with the group generated by $i\partial^2_x$.
Second, the nonlinear coupling term $|X|^2 X$ often leads to 
implicit numerical methods that are costly, see for instance the Crank--Nicolson scheme proposed in \cite{MR3166967}.

Therefore, we propose to discretise the stochastic Manakov equation with an exponential integrator, that we now define. 
Let $T>0$ be a fixed time horizon and consider an integer $N\geq1$. 
We define the step size by $h=T/N$ and denote discrete times by $t_n=nh$, 
for $n=0,\ldots,N$. Discretising the integral present in the mild form \eqref{mildManakov}
(by an explicit Euler step) 
as well as the random propagator (by a midpoint rule), one gets the following exponential integrator 
\begin{equation}
\label{expInt}
X^{n+1}=U_{h,n}\left(X^n+ihF(X^n)\right),
\end{equation}
where $U_{h,n}=\left( Id+\frac12 H_{h,n} \right)^{-1}\left( Id-\frac12 H_{h,n} \right)$, with $Id$ is the identity operator 
and $\displaystyle H_{h,n}=-ihI_2\partial_x^2+\sqrt{\gamma h}\sum_{k=1}^3\sigma_k\chi_k^n\partial_x$. Here, $I_2$ is the $2\times2$ identity matrix 
and $\sqrt{h}\chi_k^n=W_k((n+1)h)-W_k(nh)$, for $k=1,2,3$, are i.i.d. Wiener increments.
Since $U_{h,n}$ is an approximation of the exponential of the linear random differential operator
in \eqref{eq:Manakov}, we choose to name the scheme \eqref{expInt} exponential integrator.

The exponential integrator \eqref{expInt} thus approximates solutions to the stochastic Manakov equation \eqref{eq:Manakov}, $X(t_n)\approx X^n$, at the grid points $t_n=nh$.

Iterating the recursion \eqref{expInt}, one gets the discrete mild form for the exponential scheme 
\begin{equation}
\label{discexpInt}
X^n=\U_{h}^{n,0}X^0+ih\sum_{\ell=0}^{n-1}\U_{h}^{n,\ell}F(X^\ell), 
\end{equation}
where $\U_{h}^{n,\ell}:=U_{h,n-1}\cdot\ldots\cdot U_{h,\ell}$.

This linearly implicit method is well-defined for all $n\geq 0$, and one has that for all
$n\in\N$, $X^n\in\H^1$ (respectively $\H^2$, resp $\H^6$) provided that $X^0\in\H^1$
(resp. $\H^2$, resp. $\H^6$). Moreover, if one assumes that $F$ is bounded by $M$ on $\H^1$,
then one has almost surely for all $n\in\N$ and $h>0$ such that $nh\leq T$,
$\norm{X^n}_1\leq \norm{X^0}_1 +T M$.

In the following, in the proofs of our results,
we denote by $C$ a positive constant that may change from one line to the other,
but that does not depend on the parameters indicated in the results' statements.

\section{Convergence analysis of the exponential integrator}
This section presents the main results of the article and gives the corresponding proofs. 
We start by considering the stochastic Manakov equation \eqref{mildManakov},
where the nonlinearity $F$ is assumed to be globally Lipschitz-continuous. 
We show strong order of convergence $1/2$ for the exponential scheme \eqref{expInt} in that case.
Then, we analyse the case of a cubic nonlinearity, i.\,e. $F(X)=|X|^2X$,
which is of course {\it not} globally Lipschitz-continuous,
and we show order of convergence in probability $1/2$, 
as well as order of convergence $1/2-$ almost surely, for the exponential scheme \eqref{expInt}. 
The main steps of the proofs use similar arguments as in \cite{MR3166967,MR3312594,MR3736655}
as well as other works on the numerical analysis of SPDEs. However, some technical details
are handled differently in this paper (see for example the estimation of $J_3^n$ and $J_4^n$ below).

\subsection{The Lipschitz-continuous case}
We present a strong convergence analysis of the exponential integrator \eqref{expInt} when applied to the stochastic Manakov equation \eqref{mildManakov} when 
$F$ is globally Lipschitz-continuous on $\H^1$.
This is the case, for instance,
when one introduces a cut-off function for the cubic nonlinearity present in \eqref{eq:Manakov}: 
Let $R>0$ and $\theta\in\mathcal{C}^\infty(\R_+)$, with $\theta\geq0$, $\text{supp}(\theta)\subset[0,2]$ and $\theta\equiv1$ on $[0,1]$. 
For $x\geq0$, we set $\theta_R(x)=\theta(\frac xR)$ and define $F_R(X)=\theta_R(\norm{X}^2_{\H^1})|X|^2X$. 
We thus obtain a bounded globally Lipschitz-continuous function $F_R$ from $\H^1$ to $\H^1$,
which sends bounded subsets of $\H^2$ to bounded subsets from $\H^2$, 
resp. of $\H^6$ to $\H^6$.
For ease of presentation,
we denote the stochastic processes $X(t)$ and $X^n$ solutions to the continuous
problem and to the discrete problem, instead of using the notation
$X_R(t)$ and $X_R^n$ that we will use later in the paper, to point out the difference between
truncated and untruncated problems and solutions.

\begin{theorem}\label{thm-strong}
Let $T\geq0$, $p\geq1$, and $X^0\in\H^6$. Consider a bounded Lipschitz nonlinearity $F$, defined as above, in the stochastic Manakov equation \eqref{mildManakov}. 
Then, there exist a constant $C=C(F,\gamma, T,p,\norm{X^0}_{\H^6})$ such that the exponential integrator \eqref{expInt} has strong order of convergence $1/2$: There exists $h_0>0$ such that
$$
\forall h\in (0,h_0),\qquad
\E\bigl[ \max_{n=0,1,\ldots, N}\norm{X^n-X(t_n)}_{\H^1}^{2p} \bigr]\leq C h^p.
$$
\end{theorem}
\begin{proof}
  Let us denote the difference $X^n-X(t_n)$ by $e^n$. 
  Using the definitions of the numerical and exact solutions, we thus obtain
\begin{align*}
\norm{e^n}_{1}&=\norm{\U_{h}^{n,0}X^0+ih\sum_{\ell=0}^{n-1}\U_{h}^{n,\ell}F(X^\ell)-U(t_n,0)X^0-i\int_0^{t_n}U(t_n,s)F(X(s))\,\text ds}_1\\
&\leq \norm{\left(\U_{h}^{n,0}-U(t_n,0)\right)X^0}_1+\norm{\sum_{\ell=0}^{n-1}\int_{t_\ell}^{t_{\ell+1}}\left(\U_{h}^{n,\ell}F(X^\ell)-U(t_n,s)F(X(s))\right)\,\text ds}_1\\
&=:I_1^n+I_2^n.
\end{align*}
We begin by estimating the term $I_2^n$ using
\begin{align*}
I_2^n&=\left\lVert\sum_{\ell=0}^{n-1}\int_{t_\ell}^{t_{\ell+1}}\Bigl(\U_{h}^{n,\ell}F(X^\ell)-\U_{h}^{n,\ell}F(X(t_\ell))+\U_{h}^{n,\ell}F(X(t_\ell))
-\U_{h}^{n,\ell}F(X(s))+\U_{h}^{n,\ell}F(X(s))\right.\\
&\quad\left.-U(t_n,t_\ell)F(X(s))+U(t_n,t_\ell)F(X(s))-U(t_n,s)F(X(s))\Bigr)\,\text ds\right\rVert_1\\
&\leq\norm{\sum_{\ell=0}^{n-1}\int_{t_\ell}^{t_{\ell+1}}\U_{h}^{n,\ell}\left(F(X^\ell)-F(X(t_\ell))\right)\,\text ds}_1
+\norm{\sum_{\ell=0}^{n-1}\int_{t_\ell}^{t_{\ell+1}}\U_{h}^{n,\ell}\left(F(X(t_\ell))-F(X(s)))\right)\,\text ds}_1\\
&\quad+\norm{\sum_{\ell=0}^{n-1}\int_{t_\ell}^{t_{\ell+1}}\left( \U_{h}^{n,\ell}-U(t_n,t_\ell) \right)F(X(s)) \,\text ds}_1
+\norm{\sum_{\ell=0}^{n-1}\int_{t_\ell}^{t_{\ell+1}}\left( U(t_n,t_\ell)-U(t_n,s) \right)F(X(s))\,\text ds}_1\\
&=:J_1^n+J_2^n+J_3^n+J_4^n.
\end{align*}
We next bound the expectation of each of the four terms above to the power $2p$. 
Using the fact that the nonlinearity $F$ is globally Lipschitz-continuous from $\H^1$ to $\H^1$, 
and that $\U^{n,\ell}_h$ is an isometry on all $\H^s$ (see Appendix~\ref{appd}),  
the first term can be estimated as follows
\begin{align*}
\E\left[\max_{n=0,1,\ldots,N}(J_1^n)^{2p} \right]&\leq \E\left[ \max_{n=0,1,\ldots,N}\left( \sum_{\ell=0}^{n-1}
\int_{t_\ell}^{t_{\ell+1}}\,\text ds\norm{\U_{h}^{n,\ell}\left(F(X^\ell)-F(X(t_\ell))\right)}_1\right)^{2p} \right] \\
&\leq CT^{2p}\E\left[ \max_{\ell=0,1,\ldots,N}\norm{X^\ell-X(t_\ell)}_1^{2p} \right]=CT^{2p}
\E\left[\max_{\ell=0,1,\ldots,N}\norm{e^\ell}_1^{2p}\right].
\end{align*}
Similarly, for the second term we obtain
\begin{align*}
\E\left[\max_{n=0,1,\ldots,N}(J_2^n)^{2p} \right]&\leq C\E\left[ \max_{n=0,1,\ldots,N}\left( \sum_{\ell=0}^{n-1}
\int_{t_\ell}^{t_{\ell+1}}\norm{X(t_\ell)-X(s)}_1\,\text ds\right)^{2p}\right]\\
&\leq C\E\left[ \left(\sum_{\ell=0}^{N-1}\sup_{t_\ell\leq s\leq t_{\ell+1}}\norm{X(t_\ell)-X(s)}_1h \right)^{2p} \right].
\end{align*}
Using H\"older's inequality, one then gets 
\begin{align*}
\E\left[\max_{n=0,1,\ldots,N}(J_2^n)^{2p} \right]&\leq Ch^{2p} \E\left[ \left(\sum_{\ell=0}^{N-1} 1^{2p/(2p-1)} \right)^{2p-1} 
\sum_{\ell=0}^{N-1}\left( \sup_{t_\ell\leq s\leq t_{\ell+1}}\norm{X(t_\ell)-X(s)}_1 \right)^{2p} \right]\\
&\leq Ch^{2p}N^{2p-1}\sum_{\ell=0}^{N-1}\E\left[\sup_{t_\ell\leq s\leq t_{\ell+1}}\norm{X(t_\ell)-X(s)}_1^{2p} \right]
\leq Ch^{2p}N^{2p}h^p\leq Ch^p,
\end{align*}
where we have used the estimate from Lemma~5.4 (temporal regularity of the mild solution) in \cite{MR3166967}. 

Using Lemma~5.3 (uniform boundedness of the mild solution in $\H^6$) in \cite{MR3166967},
as well as the fact that $F$ sends bounded sets of $\H^6$ to bounded sets of $\H^6$,
we infer, using Proposition~2.2 (strong convergence for linear problems, i.\,e. when $F\equiv0$)
in \cite{MR3166967} that one has
\begin{equation*}
\forall s\in[0,T],\qquad
\E\left(\max_{n\in\{0,\dots,N\}}\max_{\ell\in\{0,\dots,n\}}\left\|\left(\U_h^{n,\ell}-U(t_n,t_\ell)\right) F(X(s))\right\|_1^{2p} \right)\leq C h^p.
\end{equation*}
Therefore, we estimate the third term, using H\"older's inequality, as follows
\begin{align*}
  \E\left[\max_{n=0,1,\ldots,N}(J_3^n)^{2p} \right]
  &\leq\E\left[\max_{n=0,1,\ldots,N}\left( \sum_{\ell=0}^{n-1}
    \int_{t_\ell}^{t_{\ell+1}}\norm{\left(\U_h^{n,\ell}-U(t_n,t_\ell)\right)F(X(s))}_1\,\text ds \right)^{2p}\right]\\
  &\leq\E\left[\left(\int_{0}^{T}\max_{n=0,1,\ldots,N}\max_{\ell=0,\ldots,n}
                                                                                                                     \norm{\left(\U_h^{n,\ell}-U(t_n,t_\ell)\right)
                                                                                                                     F(X(s))}_{1}\text ds \right)^{2p} \right]\\
                                                 & \leq C T^{2p-1} \E\left[
                                                          \int_{0}^{T}
                                                   \max_{n=0,1,\ldots,N}\max_{\ell=0,\ldots,n}
                                                          \norm{\left(\U_h^{n,\ell}-U(t_n,t_\ell)\right)F(X(s))}_1^{2p}         \dd s           \right]\\
&\leq CT^{2p-1}\int_0^T\E\left[\max_{n=0,1,\ldots,N}\max_{\ell=0,1,\ldots,n}\norm{\left(\U_h^{n,\ell}-U(t_n,t_\ell)\right)F(X(s))}_1^{2p} \right]\dd s\leq Ch^p.
\end{align*}
To bound the last term, we first use the isometry property of the continuous random propagator 
and H\"older's inequality to get
\begin{align}\label{truc2ouf}
\E\left[\max_{n=0,1,\ldots,N}(J_4^n)^{2p} \right]&\leq\E\left[ \max_{n=0,1,\ldots,N}\left( \sum_{\ell=0}^{n-1}\int_{t_\ell}^{t_{\ell+1}}
\norm{\left(U(t_n,t_\ell)-U(t_n,s)\right)F(X(s))}_1\,\text ds \right)^{2p} \right]\nonumber\\
&\leq\E\left[ \max_{n=0,1,\ldots,N}\left( \sum_{\ell=0}^{n-1}\int_{t_\ell}^{t_{\ell+1}} 1\times
\norm{(Id-U(s,t_\ell))F(X(s))}_{1} \,\text ds\right)^{2p} \right]\nonumber\\
&\le Ch^{2p}
	\E\left[ \left( \left(\sum_{\ell=0}^{N-1}1^{\frac{2p}{2p-1}}\right)^{\frac{2p-1}{2p}}\left(\sum_{\ell=0}^{N-1}\sup_{t_\ell\leq s\leq t_{\ell+1}}\norm{\left(Id-U(s,t_\ell)\right)F(X(s))}^{2p}_1\right)^{\frac1{2p}} \right)^{2p}\right]\nonumber\\
	&\leq Ch^{2p}N^{2p-1}\sum_{\ell=0}^{N-1}\E\left[ \sup_{t_\ell\leq s\leq t_{\ell+1}}\norm{\left(Id-U(s,t_\ell)\right)F(X(s))}^{2p}_1\right].
\end{align}
In order to estimate the expectation above, we write this term as 
	\begin{align}\label{truc2ouf2}
	\E&\left[ \sup_{t_\ell\leq s\leq t_{\ell+1}}\norm{\left(Id-U(s,t_\ell)\right)( F(X(t_\ell))-F(X(t_\ell))+F(X(s)) )}^{2p}_1\right]\nonumber\\
	&\leq 
	C\E\left[ \sup_{t_\ell\leq s\leq t_{\ell+1}}\norm{\left(Id-U(s,t_\ell)\right) F(X(t_\ell))}^{2p}_1\right]+
        C\E\left[ \sup_{t_\ell\leq s\leq t_{\ell+1}}\norm{\left(Id-U(s,t_\ell)\right)\left( F(X(s))-F(X(t_\ell)) \right)}^{2p}_1\right].
	\end{align}
    
    The first term in the equation above is the exact solution to the linear SPDE 
	$\displaystyle i\text dZ(t)+\frac{\partial^2Z(t)}{\partial x^2}\text dt+i\sqrt{\gamma}\sum_{k=1}^3\sigma_k\frac{\partial Z(t)}{\partial x}\circ\,\text d W_k(t)=0$ with initial value $F(X(t_\ell))$ 
	at initial time $t_\ell$ which has the mild Ito form 
	$$
	Z(t)-F(X(t_\ell))=\left(S(t-t_\ell)-Id\right)F(X(t_\ell))+i\sqrt{\gamma}\sum_{k=1}^3\int_{t_\ell}^tS(t-u)\sigma_k\partial_xZ(u)\,\text dW_k(u),
	$$
	where $S(t)$ is the group solution to the free Schr\"odinger equation. 
Owing at the regularity property of the group $S$ (see for instance the first inequality in the proof of \cite[Lemma 4.2.1]{GazeauPhd}), 
the fact that the exact solution $X$ is almost surely bounded in $\H^2$,
that $F$ sends bounded sets of $\H^2$ to bounded sets of $\H^2$, 
	and Burkholder--Davis--Gundy's inequality (for the second term), one obtains the following bound for the first term in \eqref{truc2ouf2}
    \begin{align*}
    \E\left[ \sup_{t_\ell\leq s\leq t_{\ell+1}}\norm{\left(Id-U(s,t_\ell)\right)F(X(t_\ell))}^{2p}_1\right]\leq Ch^p. 
    \end{align*}
    Using the fact that the random propagator $U$ is an isometry, 
    that $F$ is globally Lipschitz-continuous, and the regularity property of the exact solution $X$, one gets the estimate 
    $$
    \E\left[ \sup_{t_\ell\leq s\leq t_{\ell+1}}\norm{\left(Id-U(s,t_\ell)\right)\left( F(X(s))-F(X(t_\ell)) \right)}^{2p}_1\right]\leq 
    \E\left[ \sup_{t_\ell\leq s\leq t_{\ell+1}}\norm{F(X(s))-F(X(t_\ell))}^{2p}_1\right]\leq Ch^{2p},
    $$
    for the second term in \eqref{truc2ouf2}. 
    
    Combining the estimates above, one finally arrives at the bound 
\begin{align*}
\E\left[\max_{n=0,1,\ldots,N}(J_4^n)^{2p} \right]\leq Ch^{2p}N^{2p-1}\sum_{\ell=0}^{N-1}h^p\leq Ch^p.
\end{align*}
Altogether we thus obtain
\begin{align*}
\E\left[ \max_{n=0,1,\ldots,N}\norm{e^n}_1^{2p} \right]&\leq 
C\E\left[\max_{n=0,1,\ldots,N}\left(I_1^n\right)^{2p} \right]+
CT^{2p}\E\left[\max_{n=0,1,\ldots,N}\norm{e^n}_1^{2p}\right]+Ch^p\\
&\leq Ch^p+CT^{2p}\E\left[ \max_{n=0,1,\ldots,N}\norm{e^n}_1^{2p} \right],
\end{align*}
using once again \cite[Proposition~2.2]{MR3166967}.

For $T=T_1$ small enough, i.\,e. such that $CT_1^{2p}<1$, the inequality above gives 
\begin{align*}
\E\left[ \max_{n=0,1,\ldots,N}\norm{e^n}_1^{2p} \right]&\leq\frac{C}{1-CT_1^{2p}}h^p,
\end{align*}
on $[0,T_1]$.
In order to iterate this procedure, we impose, if necessary, that $h$ is small enough
(or, equivalently, that $N$ is big enough), to ensure that $T_1$ can be chosen as before and as
some integer multiple of $h$ (say $T_1=rh$ for some positive integer $r$), while $T$ is some
multiple integer of $T_1$ (say $MT_1=T$ for some positive integer $M$).
To obtain a bound for the error on the longer time interval $[0,T]$, we iterate the procedure above
by choosing $T_2=2T_1$ and estimate the error on the interval $[T_1,T_2]$.
We repeat this procedure, $M$ times, up to final time $T$.
This can be done since the above error estimates are uniform on the intervals $[T_m,T_{m+1}]$ 
for $m=0,\ldots,M-1$ (with a slight abuse of notation for the time interval):  
\begin{align*}
\E\left[ \max_{[T_m,T_{m+1}]}\norm{X^n-X_m(t_n)}_1^{2p} \right]&\leq C_Eh^p,
\end{align*}
where $C_E$ is the error constant obtained above, $t_n=nh$ are discrete times in $[T_m,T_{m+1}]$, $X_0(t):=X(t)$ is the exact solution with initial value $X^0$, 
$X_m(t)$ denotes the exact solution with initial value $Y^m$ at time $T_m=mT_1=(mr)h=t_{mr}$, 
and $Y^m=X_{mr}$ corresponds to numerical solutions at time $T_m$ 
for $m=0,\ldots,M-1$. 
For the total error, we thus obtain (details are only written for the first two intervals)
\begin{align*}
\E\left[ \max_{n=0,1,\ldots,N}\norm{e^n}_1^{2p} \right]&=\E\left[ \max_{[0,T]}\norm{X^n-X(t_n)}_1^{2p} \right]
\lesssim\E\left[ \max_{[0,T_1]}\norm{X^n-X(t_n)}_1^{2p} \right]\\
&\quad+\E\left[ \max_{[T_1,T_2]}\norm{X^n-X(t_n)}_1^{2p} \right]+
\ldots+\E\left[ \max_{[T_{M-1},T_M]}\norm{X^n-X(t_n)}_1^{2p} \right]\\
&\lesssim C_Eh^p+\E\left[ \max_{[T_1,T_2]}\norm{X^n-X_1(t_n)}_1^{2p} \right]+\E\left[ \max_{[T_1,T_2]}\norm{X_1(t_n)-X(t_n)}_1^{2p} \right]+\ldots\\
&\lesssim C_Eh^p+C_E h^p+C_L\E\left[ \norm{Y^1-X(T_1) }_1^{2p} \right]+\ldots\lesssim C_Eh^p+C_LC_Eh^p+\ldots \\
&\lesssim C_Eh^p+C_LC_Eh^p+C_L^2C_Eh^p+\ldots+C_L^{M-1}C_Eh^p \lesssim Ch^p, 
\end{align*}
where $C_L$ is the Lipschitz constant of the exact flow of \eqref{eq:Manakov} from $\H^1$ to itself
and the last constant $C$ is independent of $N$ and $h$ with $Nh=T$ for $N$ big enough. 
This concludes the proof of the theorem.
\end{proof}

\subsection{Convergence in the non-Lipschitz case}
Using the above result as well as ideas from \cite{MR1873517,MR2268663,MR3166967,MR3312594,MR3736655}, 
one can show convergence in probability of order $1/2$ and almost sure convergence of order $1/2-$ 
for the exponential integrator \eqref{expInt} when applied to the stochastic Manakov equation \eqref{eq:Manakov}.

\begin{proposition}\label{propProba}
  Let $X^0\in\H^6$ and $T>0$.
  Denote by $\tau^*=\tau^*(X_0,\omega)$ the maximum stopping time for the existence
  of a strong adapted solution,
  denoted by $X(t)$, of the stochastic Manakov equation \eqref{eq:Manakov}.
  For all stopping time $\tau<\tau^*\wedge T$ a.s. there exists $h_0>0$ such that we have 
  \begin{equation*}
    \forall h\in (0,h_0),\qquad
  \lim_{C\to\infty}
    \mathbb{P}\left( \max_{0\leq n\leq N_{\tau}} \norm{X^n-X(t_n)}_{\H^1}\geq C h^{1/2} \right)=0, 
  \end{equation*}
where $X^n$ denotes the numerical solution given by the exponential integrator \eqref{expInt} with time step $h$ and $N_\tau=\lceil\frac{\tau}{h}\rfloor$. 
\end{proposition}

\begin{proof}
For $R>0$, let us denote by $X_R$, resp. $X_R^n$, the exact, resp. numerical, solutions to the stochastic Manakov equation 
\eqref{mildManakov} with a truncated nonlinearity $F_R$.
We denote by $\kappa$ a positive constant
such that for all $Y\in \H^1$, $\norm{|Y|^2 Y}_1\leq \kappa\norm{Y}^3_1$.

Fix $X^0\in\H^6$, $T>0$, $\varepsilon>0$.
Let $\tau$ be a stopping time such that a.s. $\tau<\tau^*\wedge T$.
By Theorem~\ref{thmExact}, there exists an $R_0>1$ such that 
$\displaystyle\mathbb{P}\left(\sup_{t\in[0,\tau]}\norm{X(t)}_1\geq R_0-1\right)\leq\varepsilon/2$.
We have the inclusion
\begin{align*}
\left\{ \max_{0\leq n\leq N_\tau}\norm{X^n-X(t_n)}_1\geq\varepsilon  \right\}&\subset \left\{ \max_{0\leq n\leq N_\tau}\norm{X(t_n)}_1\geq R_0-1 \right\} \\
&\cup \left( \left\{ \max_{0\leq n\leq N_\tau}\norm{X^n-X(t_n)}_1\geq\varepsilon \right\} \cap \left\{ \max_{0\leq n\leq N_\tau}\norm{X(t_n)}_1<R_0-1 \right\} \right).
\end{align*}
Taking probabilities, we obtain
\begin{align*}
\PP\left( \left\{ \max_{0\leq n\leq N_\tau}\norm{X^n-X(t_n)}_1\geq\varepsilon  \right\} \right)\leq\varepsilon/2+
\PP\left( \left\{ \max_{0\leq n\leq N_\tau}\norm{X^n-X(t_n)}_1\geq\varepsilon \right\} \cap \left\{ \max_{0\leq n\leq N_\tau}\norm{X(t_n)}_1<R_0-1 \right\} \right).
\end{align*}
In order to estimate the terms on the right-hand side, we define the random variable
$n_\varepsilon:=\min\{n\in\{0,\dots,N_\tau\}\colon \norm{X^n-X(t_n)}_1\geq\varepsilon\}$,
with the convention that $n_\varepsilon=N_\tau+1$ if the set is empty.
If $\displaystyle\max_{0\leq n\leq N_\tau}\norm{X(t_n)}_1<R_0-1$ then we have by triangle inequality
$$
\max_{0\leq n\leq n_\varepsilon-1}\norm{X^n}_1
=\max_{0\leq n\leq n_\varepsilon-1}\norm{X^n-X(t_n)+X(t_n)}_1\leq\varepsilon+R_0-1\leq R_0.
$$
By definition of the exponential integrator \eqref{expInt},
for $h<\frac{3}{R_0^2{\kappa}}$, we have
\begin{equation*}
  \norm{X^{n_\varepsilon}}_1
  \leq \norm{X^{n_\varepsilon-1}}_1+h\norm{|X^{n\varepsilon-1}|^2X^{n_\varepsilon-1}}_1
  \leq R_0+\kappa h\norm{X^{n_\varepsilon-1}}_1^3\leq R_0+\kappa h R_0^3\leq 4 R_0,
\end{equation*}
in this case and thus $X^n=X_{4R_0}^n$ for $0\leq n\leq n_\varepsilon$.

If $n_\varepsilon\leq N_\tau$, then $\norm{X^{n_\varepsilon}_{4R_0}-X_{4R_0}(t_{n_\varepsilon})}_1\geq\varepsilon$ thanks to the definition of $n_\varepsilon$. 
Therefore we get $\displaystyle\max_{0\leq n\leq N_\tau}\norm{X^n_{4R_0}-X_{4R_0}(t_n)}_1\geq\varepsilon$. 
Furthermore, by definition of $n_\varepsilon$, we have 
$\displaystyle\left\{ \max_{0\leq n\leq N_\tau}\norm{X^n-X(t_n)}_1\geq\varepsilon \right\}\cap\left\{ n_\varepsilon>N_\tau \right\}=\varnothing$. 
We then deduce that 
$$
\left\{ \max_{0\leq n\leq N_\tau}\norm{X^n-X(t_n)}_1\geq\varepsilon \right\}=\left\{ \max_{0\leq n\leq N_\tau}\norm{X^n-X(t_n)}_1\geq\varepsilon \right\}
\cap\left\{ n_\varepsilon\leq N_\tau \right\}.
$$
Combining the above, using Markov's inequality as well as the strong error estimates from Theorem~\ref{thm-strong}, since $\tau<T$ a.s., there exists $C>0$ such that
\begin{align*}
&\PP\left( \max_{0\leq n\leq N_\tau}\norm{X^n-X(t_n)}_1\geq\varepsilon, n_\varepsilon\leq N_\tau,\max_{0\leq n\leq N_\tau}\norm{X(t_n)}_1<R_0-1 \right)\\
&\leq\PP\left( \max_{0\leq n\leq N_\tau}\norm{X_{4R_0}^n-X_{4R_0}(t_n)}_1\geq\varepsilon \right)
\leq\frac1{\varepsilon^{2p}}\E\left[ \max_{0\leq n\leq N_\tau}\norm{X^n_{4R_0}-X_{4R_0}(t_n)}_1^{2p} \right]\leq\frac1{\varepsilon^{2p}}Ch^p.
\end{align*}
This last term is smaller than $\varepsilon/2$ for $h$ small enough. 
All together we obtain
$$
\PP\left( \max_{0\leq n\leq N_\tau}\norm{X^n-X(t_n)}_1\geq\varepsilon \right)\leq\frac\varepsilon2+\frac\varepsilon2=\varepsilon
$$
and thus convergence in probability. 

To get the order of convergence in probability,
we choose $R_1\geq R_0-1$ such that for all $h>0$ small enough, 
$\displaystyle\PP\left( \max_{0\leq n\leq N_\tau}\norm{X^n}_1\geq R_1\right)\leq\frac\varepsilon2$.
As above, for all positive real number $C$, we have 
\begin{align*}
\left\{ \max_{0\leq n\leq N_\tau}\norm{X^n-X(t_n)}_1\geq Ch^{1/2} \right\}&\subset \left\{\max_{0\leq n\leq N_\tau}\norm{X(t_n)}_1\geq R_1 \right\}\\
&\cup\left( \left\{ \max_{0\leq n\leq N_\tau}\norm{X^n-X(t_n)}_1\geq Ch^{1/2} \right\}\cap \left\{\max_{0\leq n\leq N_\tau}\norm{X(t_n)}_1<R_1\right\} \right).
\end{align*}
Taking probabilities and using Markov's inequality as well as the strong error estimate from Theorem~\ref{thm-strong}, we obtain  
\begin{align*}
\PP\left( \left\{ \max_{0\leq n\leq N_\tau}\norm{X^n-X(t_n)}_1\geq Ch^{1/2} \right\} \right)&\leq\frac\varepsilon2+
\PP\left( \left\{ \max_{0\leq n\leq N_\tau}\norm{X_{4R_1}^n-X_{4R_1}(t_n)}_1\geq Ch^{1/2} \right\} \right)\\
&\leq\frac\varepsilon2+\frac{K(4R_1,\gamma,T,p,\norm{X_0}_6)}{C^{2p}},
\end{align*}
since $\tau\leq T$ almost surely.
For $C$ large enough, we infer
$$
\PP\left( \left\{ \max_{0\leq n\leq N_\tau}\norm{X^n-X(t_n)}_1\geq Ch^{1/2} \right\} \right)\leq\frac\varepsilon2+\frac\varepsilon2=\varepsilon,
$$
uniformly for $h<h_0$.
Finally, the order of convergence in probability of the exponential integrator is $1/2$.
\end{proof}
Using the results above, one arrives at the following proposition, which establishes
that the scheme has almost sure convergence order $1/2^-$.
\begin{proposition}
  Under the assumptions of Proposition~\ref{propProba}, for all $\delta\in\left(0,\frac12\right)$
  and $T>0$,
  there exists a random variable $K_\delta(T)$ such that for all stopping time $\tau$ with $\tau<\tau^*\wedge T$, we have
$$
\max_{n=0,\ldots, N_\tau}\norm{X^n(\omega)-X(t_n,\omega)}_{\H^1}\leq K_\delta(T,\omega)h^\delta\quad\quad\quad\mathbb{P}-a.s.,
$$
for $h>0$ small enough.
\end{proposition}
\begin{proof}
  Let $\tau$ be a stopping time such that $\tau<\tau^*\wedge T$ almost surely.
  Fix $R>0$, $p>1$ and $\delta\in(0,\frac12)$.
  Using the strong error estimate from Theorem~\ref{thm-strong} and Markov's inequality, one gets
  positive $h_0$ and $C$, which does not depend on $\tau$ itself, such that
  $$
  \forall h\in (0,h_0),\qquad
\PP\left(\max_{0\leq n\leq N_\tau}\norm{X_R^n-X_R(t_n)}_1>h^\delta\right)\leq Ch^{p(1-2\delta)}.
  $$
  Using \cite[Lemma~2.8]{MR1873517}, one then obtains that, choosing $p>1$ sufficiently
  large to ensure that $p(1-2\delta)>1$, 
  there exists a positive random variable $K_\delta(R,\gamma,T,p,\cdot)$ such that
  \begin{equation}\label{asest}
    \mathbb{P}-a.s., \quad \forall h\in(0,h_0),\qquad
  \max_{0\leq n\leq N_\tau}\norm{X_R^n-X_R(t_n)}_1 \leq K_\delta(R,\gamma,T,p,\omega)h^\delta.
\end{equation}
After this preliminary observation, we shall proceed as in the proof of Proposition~\ref{propProba}. 
We know that, since $\tau<\tau^*$ a.s., there exists a random variable $R_0$ such that
$$
\sup_{0\leq t\leq\tau}\norm{X(t)}_1\leq R_0(\omega)\quad\mathbb{P}-\text{a.s.}.
$$
Let now $\varepsilon\in(0,1)$ and $h$ small enough ($h\leq 3R_0^{-2}(\omega)\kappa^{-1}$). Assume by contradiction that 
$$
\max_{0\leq n\leq N_\tau}\norm{X^n-X(t_n)}_1\geq\varepsilon.
$$
Define $n_\varepsilon:=\min\{n\colon\norm{X^n-X(t_n)}_1\geq\varepsilon\}$. By definition of $R_0$ and $h$, we have that $\norm{X^n}_1\leq R_0$ a.s. for $0\leq n<n_\varepsilon-1$. 
Hence, $\norm{X^{n_\varepsilon}}_1\leq 4R_0$ and so the numerical solution equals to the numerical solution of the truncated equation 
$X^n=X_{4R_0}^n$ for $n=0,1,\ldots,n_\varepsilon$. 
We thus obtain that $\displaystyle\max_{0\leq n\leq N_\tau}\norm{X^n_{4R_0}-X_{4R_0}(t_n)}_1\geq\varepsilon$ for $h$ small enough. 
This contradicts \eqref{asest} with $R=4R_0$. Therefore, we have almost sure convergence. 

To get the order of almost sure convergence, we proceed similarly as in the proof of Proposition~\ref{propProba}. 
From the above, we have for $\omega$ in a set of probability one and all
$\varepsilon>0$, there exists $h_0>0$ such that for all $h\leq h_0$, 
$\displaystyle\max_{0\leq n\leq N_\tau}\norm{X^n-X(t_n)}_1\leq\varepsilon$. Thus, there exists $R_1(\omega)>R_0(\omega)$ such that 
$\displaystyle\max_{0\leq n\leq N_\tau}\norm{X^n}_1\leq R_1(\omega)$. 

If now $h\leq 3R_1^{-2}\kappa^{-1}:=h_0$, we obtain from \eqref{asest} that
\begin{align*}
\max_{0\leq n\leq N_\tau}\norm{X^n-X(t_n)}_1=\max_{0\leq n\leq N_\tau}\norm{X^n_{R_1}-X_{R_1}(t_n)}_1\leq K_\delta(R_1,\gamma,T,p,\omega)h^\delta.
\end{align*}
This shows that the order of a.s. convergence of the exponential integrator is $\frac12-$.
\end{proof}

\section{Numerical experiments}
This section presents various numerical experiments in order to illustrate the main properties of the exponential integrator \eqref{expInt}, denoted by \textsc{SEXP} below. 
We will compare this numerical scheme with the following ones: 
\begin{itemize}
\item The nonlinearly implicit Crank--Nicolson scheme from \cite{MR3166967}
\begin{equation}
\label{CN}\tag{CN}
X^{n+1}=X^n-H_{h,n}X^{n+1/2}+ihG(X^n,X^{n+1}),
\end{equation}
where $G(X^n,X^{n+1})=\frac12\left(|X^n|^2+|X^{n+1}|^2\right)X^{n+1/2}$ and $X^{n+1/2}=\frac{X^{n+1}+X^n}2$.
\item The Lie--Trotter splitting scheme presented in \cite{MR3166967}
\begin{equation}
\label{LT}\tag{LT}
X^{n+1}=U_{h,n+1}\left(X^n+i\int_{t_n}^{t_{n+1}}F(Y^n(s))\,\text ds\right),
\end{equation}
where we recall that $F(X)=|X|^2X$, $Y^n$ is the exact solution to the nonlinear differential
equation $i\text dY^n+F(Y^n)\,\text dt=0$ with the initial condition $Y^n(t_n)=X^n$. 
A convergence analysis of this time integrator is under investigation in \cite{bcd20}. 
\item The relaxation scheme presented in \cite{MR3166967}
\begin{equation}
\label{relax}\tag{Relax}
i\left(X^{n+1}-X^n\right)+H_{h,n}\left(\frac{X^{n+1}+X^n}{2} \right)+\Phi^{n+1/2}\left(\frac{X^{n+1}+X^n}{2} \right)=0,
\end{equation}
where $\Phi^{n+1/2}=2|X^n|^2-\Phi^{n-1/2}$ with $\Phi^{-1/2}=|X^0|^2$.
\end{itemize}
We will consider the SPDE \eqref{eq:Manakov} on an interval $[-a,a]$ with a sufficiently large $a>0$ with homogeneous Dirichlet boundary conditions. 
The spatial discretisation is done by centered finite differences with mesh size denoted by $\Delta x$. Unless stated otherwise, the initial condition for the SPDE is the soliton 
of the deterministic Manakov equation \cite{HASEGAWA2004241} given by
\begin{equation}
X^0=X(x,0)=\begin{pmatrix}
\cos(\theta/2)\exp({i\phi_1})\\
\sin(\theta/2)\exp({i\phi_2})
\end{pmatrix}
\eta\sech(\eta x)\exp({-i\kappa(x-\tau)+i\alpha}),
\end{equation}
with the parameters $\alpha=\tau=\phi_1=\phi_2=\kappa=0$ and $\theta=\pi/4$, $\eta=1$.

\subsection{Evolution plots}
In the first numerical experiment, we solve the stochastic Manakov equation\eqref{eq:Manakov} with $a=50$ on the time interval $[0,3]$ and 
discretisation parameters $h=3/625$ and $\Delta x=1/4$. Figure~\ref{fig:evo} displays the space-time evolution 
of the numerical intensities $|X^1|^2$ and $|X^2|^2$
along solutions given by the exponential integrator \eqref{expInt}. 
An energy exchange due to the stochastic perturbation 
and the nonlinearity can be observed. This produces small amplitude perturbations at the basis of the soliton, 
leading to the formation of further solitons. 

\begin{figure}[h]
\centering
\includegraphics*[width=0.48\textwidth,keepaspectratio]{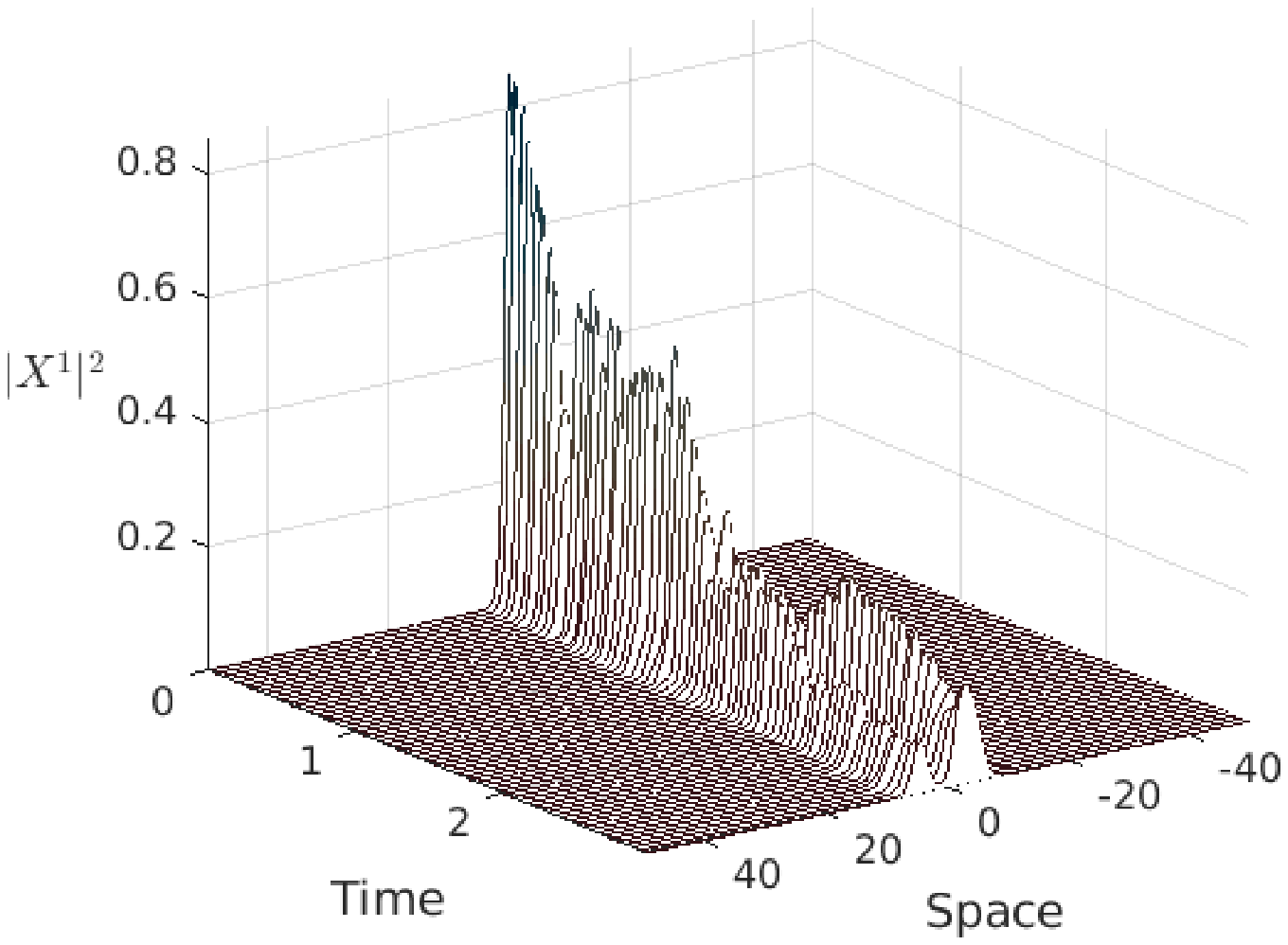}
\includegraphics*[width=0.48\textwidth,keepaspectratio]{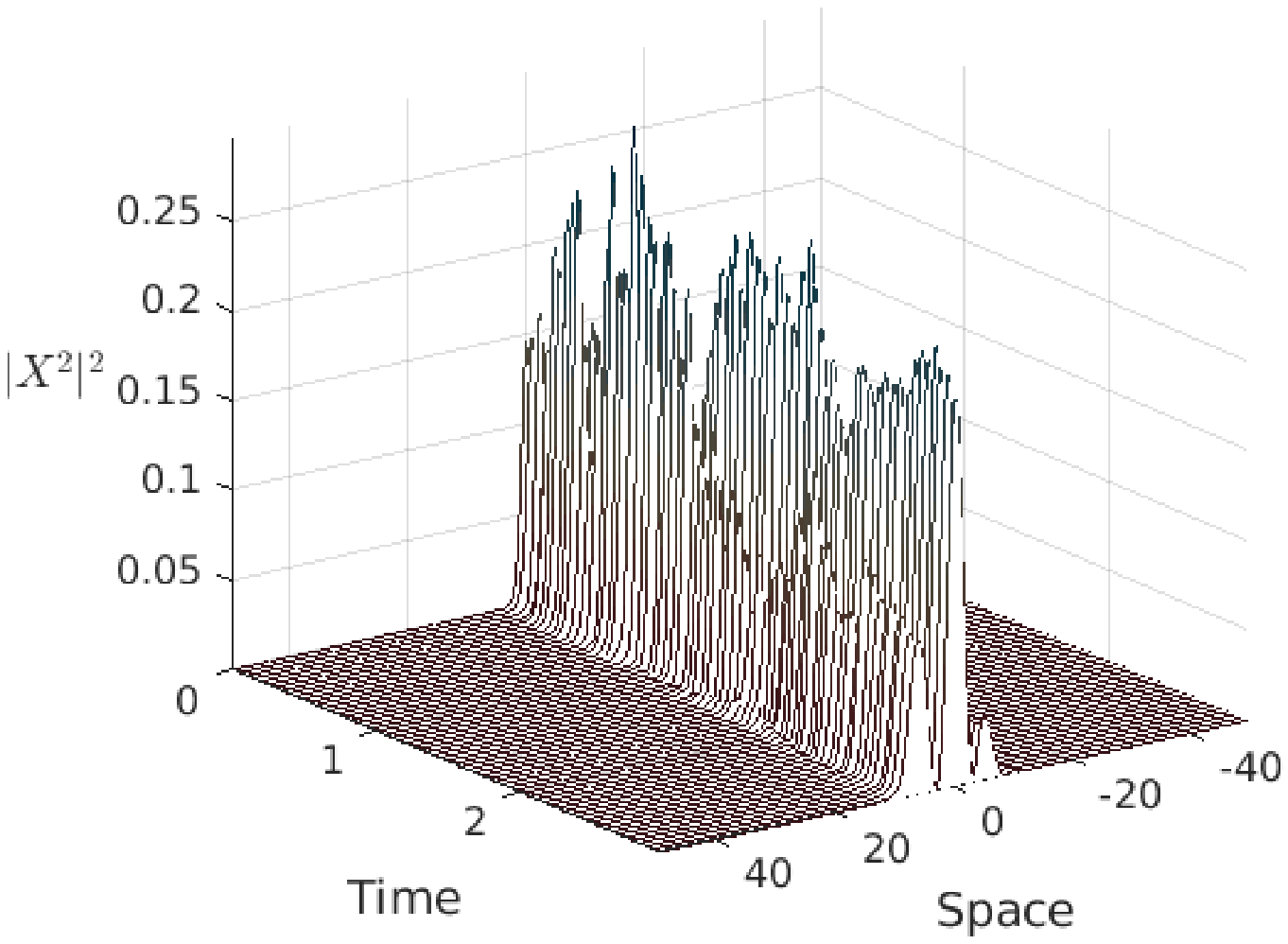}
\caption{Space-time evolution of the intensity of the first component (left) and the second component (right).}
\label{fig:evo}
\end{figure}

\subsection{Strong convergence}
In order to illustrate the strong rate of convergence of the exponential integrator \eqref{expInt} 
stated in Theorem~\ref{thm-strong}, we discretise the stochastic Manakov equation \eqref{eq:Manakov} 
with $a=50$ and mesh size $\Delta x=0.4$. 
We compute the errors 
$\E\left[\norm{X^N-X_{\text{ref}}(T)}_{\H^1}^2\right]$ at the time $T=1$ for time steps ranging from $h=2^{-13}$ to $h_{\text{ref}}=2^{-19}$ 
and report these in Figure~\ref{fig:strong}. The reference solution is computed using the exponential integrator 
and the expected values are approximated by computing averages over $M_s=250$ samples. 
We observed that using a larger number of samples ($M_s=500$) does not significantly improve 
the behaviour of the convergence plots (the results are not displayed). 

\begin{figure}[h]
\centering
\includegraphics*[height=7cm,keepaspectratio]{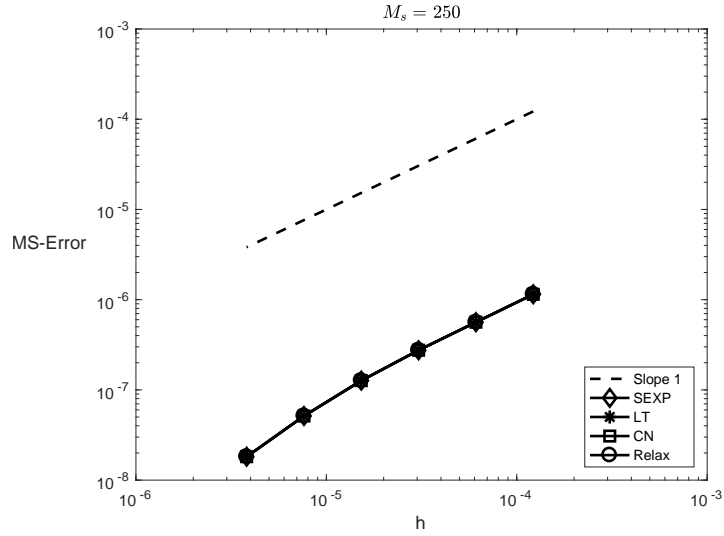}
\caption{Strong rates of convergence.}
\label{fig:strong}
\end{figure}

\subsection{Computational costs}
The goal of this numerical experiment is to compare the computational cost of the exponential integrator introduced in this paper to that 
of numerical methods from the literature. We run all numerical schemes over the time interval $[0,0.5]$ for the stochastic Manakov equation 
\eqref{eq:Manakov} with $\gamma=1$. 
We discretise the spatial domain with $a=50$, using a mesh of size $\Delta x=0.2$. 
We run $500$ samples for each numerical scheme. For each scheme and each sample, we run several time steps and compare the $\L^2$ error 
at the final time with a reference solution provided for the same sample by the same scheme for a very small time step $h=2^{-16}$. 
Figure~\ref{fig:compcos} displays the total computational time for all the samples, for each numerical scheme and each time step, as a function 
of the averaged final error. 
One observes that the performance of the Crank--Nicolson scheme is a little bit inferior than the performance for the other numerical schemes. 

\begin{figure}[h]
\centering
\includegraphics*[height=7cm,keepaspectratio]{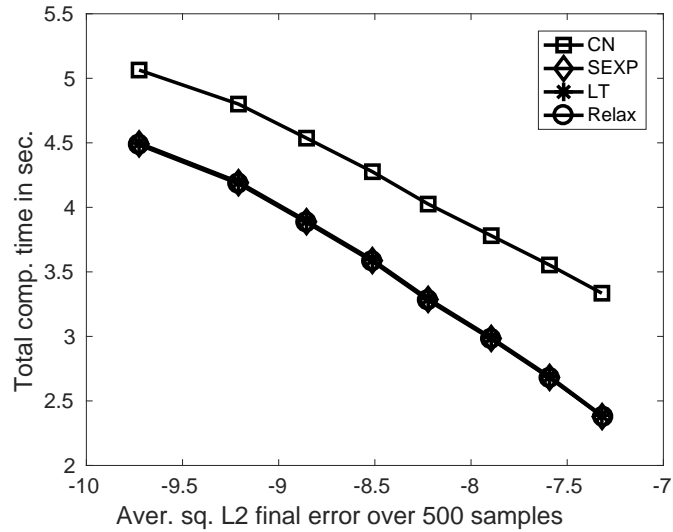}
\caption{Computational time as a function of the averaged final error for the four numerical methods.}
\label{fig:compcos}
\end{figure}

\subsection{Preservation of the $\L^2$-norm}
The next numerical experiment illustrates the preservation of the $\L^2$-norm along one sample path of the above numerical schemes. 
For this, we consider $a=50$, $\gamma=1$, time interval $[0,3]$ and discretisation parameters $h=0.006$ and $\Delta x=0.25$. 
The results are displayed in Figure~\ref{fig:L2}. Exact preservation of the $\L^2$-norm for the Crank--Nicolson, 
the Lie--Trotter and the relaxation schemes is observed. 
A small drift is observed for the exponential scheme. 

\begin{figure}[h]
\centering
\includegraphics*[height=7cm,keepaspectratio]{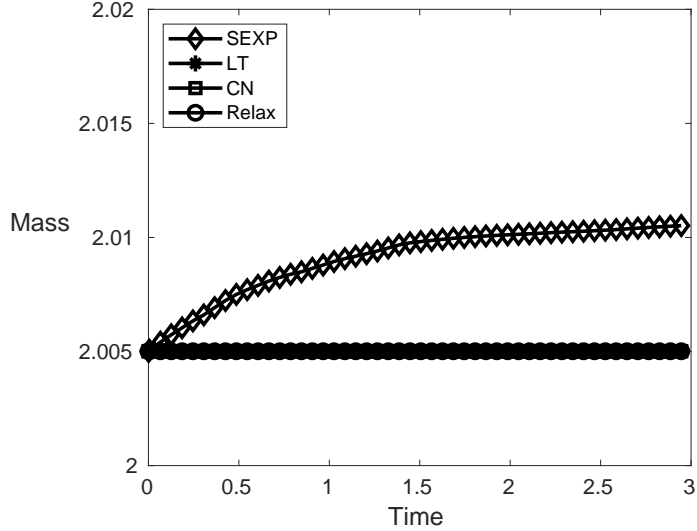}
\caption{Evolution of the $\L^2$-norm along numerical solutions ($h=0.006$ and $\Delta x=0.25$).}
\label{fig:L2}
\end{figure}

\subsection{$\L^2$-preserving exponential integrators}
As seen above, the proposed exponential integrator unfortunately does not preserve the $\L^2$-norm. 
This can be fixed using ideas from \cite{MR2413146,MR3736655}. 
We thus propose the following modified exponential method for the numerical discretisation
the stochastic Manakov equation \eqref{eq:Manakov} 
\begin{align}
\label{symexp}
F_*&=F\left(U_{h,n}X^n+i\frac{h}2F_*\right)\nonumber\\
X^{n+1}&=U_{h,n}X^n+ihF_*,\tag{modEXP}
\end{align}
where we define $F(X)=|X|^2X$ for the nonlinearity.

As seen in the introduction, the exact solution to the stochastic Manakov equation \eqref{eq:Manakov} 
preserves the $\L^2$-norm. The following proposition states that the modified exponential method enjoys 
the same property. 
\begin{proposition}
The exponential integrator \eqref{symexp} preserves the $\L^2$-norm. 
\end{proposition}
\begin{proof}
By definition of the exponential integrator and using the isometry property of the discrete random propagator $U_{h,0}$, one obtains
$$
\norm{X^1}^2=\norm{U_{h,0}X^0+ihF_*}^2=\norm{X^0}^2+h^2\norm{F_*}^2+h(U_{h,0}X^0,iF_*)_2+h(iF_*,U_{h,0}X^0)_2.
$$
Setting $Y:=U_{h,0}X^0+i\frac{h}2F_*$ or $U_{h,0}X^0=(Y-i\frac h2F_*)$, one gets 
\begin{align*}
\norm{X^1}^2&=\norm{X^0}^2+h^2\norm{F_*}^2+h(Y-i\frac h2F_*,iF_*)_2+h(iF_*,Y-i\frac h2F_*)_2\\
&=\norm{X^0}^2+h^2\norm{F_*}^2+h\left((Y,iF_*)_2+(iF_*,Y)_2\right)-\frac{h^2}2\norm{F_*}-\frac{h^2}2\norm{F_*}\\
&=\norm{X^0}^2+2h\mathrm{Re}\left((Y,iF_*)_2\right)=\norm{X^0}^2+0=\norm{X^0}^2
\end{align*}
since the $\L^2$-norm is an invariant for the original problem and $F_*=F(Y)$.
\end{proof}
We now numerically illustrate this property with the same parameters as in the previous numerical experiment. 
Figure~\ref{fig:L2b} shows the exact preservation of the $\L^2$-norm by the exponential scheme \eqref{symexp}. 

\begin{figure}[h]
\centering
\includegraphics*[height=7cm,keepaspectratio]{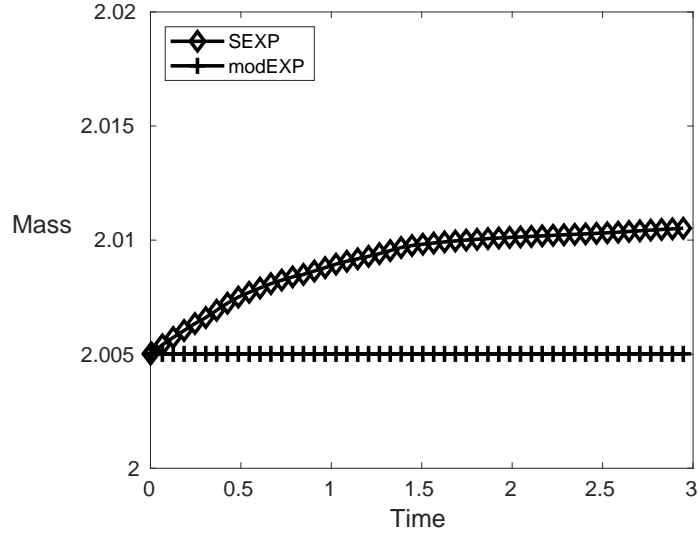}
\caption{Evolution of the $\L^2$-norm along numerical solutions given by both exponential schemes ($h=0.006$ and $\Delta x=0.25$).}
\label{fig:L2b}
\end{figure}

It would be of interest to prove the orders of convergence of the $\L^2$-preserving exponential integrator \eqref{symexp}. 
This is however out of the scope of this publication since it seems that one would need 
to use other techniques than that used in the proofs of the proposed explicit exponential integrator
\eqref{expInt}. 

\section{Appendix}\label{appd}
We prove here that the operator $U_{h,n}$ defined after equation \eqref{expInt}
is an isometry from $\H^m(\R)$ to itself for all $h>0$ and all realization
of the random variable.

\begin{proposition}
  Let $m\in\N$, $(u_0,v_0)\in \H^m(\R)$, $h>0$, $\chi_1,\chi_2,\chi_3\in\R$ and define a distribution $(u_1,v_1)\in({\mathcal S}'(\R))^2$
  as the solution of
  \begin{equation}
    \label{eq:relationU}
  \begin{pmatrix}
    u_1\\ v_1
  \end{pmatrix}
  = U_{h,n}
  \begin{pmatrix}
    u_0\\ v_0
  \end{pmatrix}.
\end{equation}
One has $(u_1,v_1)\in \H^m(\R)$ and $\norm{(u_1,v_1)}_{\H^m(\R)}=\norm{(u_0,v_0)}_{\H^m(\R)}$.
\end{proposition}

\begin{proof}
The operator $H_{h,n}$ defined after \eqref{expInt} acts on the Fourier
transform of the couples of functions at frequency $\xi\in\R$ {\it via}
the complex-valued $2\times 2$ matrix
$ih\xi^2 I_2 + i\xi\sqrt{\gamma h} (\chi_1\sigma_1+\chi_2\sigma_2+\chi_3\sigma_3)$.
This matrix reads $iS(\xi)$ where $S(\xi)$ is an hermitian $2\times 2$ matrix.
Therefore, the matrix $S(\xi)$ is diagonalizable in an orthonormal basis of $\C^2$ with
real eigenvalues $\lambda_1(\xi)$ and $\lambda_2(\xi)$. We infer that there exists
a unitary matrix $P(\xi)$ such that $S(\xi)=P(\xi)^\star D(\xi) P(\xi)$,
where $D(\xi)$ is the
diagonal matrix with $\lambda_1(\xi)$ and $\lambda_2(\xi)$ on the diagonal.
Hence, relation \eqref{eq:relationU} is equivalent to
\begin{equation*}
  \forall \xi\in\R,\qquad
  P(\xi)
  \begin{pmatrix}
    \hat u_1(\xi)\\ \hat v_1 (\xi)
  \end{pmatrix}
  =
  \begin{pmatrix}
    \frac{1-i\lambda_1(\xi)/2}{1+i\lambda_1(\xi)/2} & 0 \\
    0 & \frac{1-i\lambda_2(\xi)/2}{1+i\lambda_2(\xi)/2}
  \end{pmatrix}
  P(\xi)
  \begin{pmatrix}
    \hat u_0(\xi)\\ \hat v_0 (\xi)
  \end{pmatrix}.
\end{equation*}
Since the diagonal elements in the diagonal matrix above have modulus 1
and $P(\xi)$ is unitary,
we infer that
\begin{equation*}
  \forall \xi\in\R,\qquad
  |\hat u_1(\xi)|^2 + |\hat v_1(\xi)|^2 = |\hat u_0(\xi)|^2 + |\hat v_0(\xi)|^2.
\end{equation*}
This proves that $(u_1,v_1)\in \H^m(\R)$ since $(u_0,v_0)\in \H^m(\R)$,
and the $\H^m(\R)$-norm of these two couples of functions is the same.
\end{proof}


\bibliographystyle{plain}
\bibliography{biblio}

\end{document}